\documentclass[12pt]{amsart}
\usepackage{amsfonts,amssymb,amscd,amsmath,enumerate,color}
\def\NZQ{\mathbb}               

\def\ZZ{{\NZQ Z}}

\def\frk{\mathfrak}               

\def\mm{{\frk m}}

\def\Phi{{\frk N}}
\def\eb{{\bold e}}

\def\opn#1#2{\def#1{\operatorname{#2}}} 
\opn\chara{char} 
\opn\length{\ell} 
\opn\pd{pd} 
\opn\rk{rk}
\opn\projdim{proj\,dim} 
\opn\injdim{inj\,dim} 
\opn\rank{rank}
\opn\depth{depth} 
\opn\grade{grade} 
\opn\height{height}
\opn\embdim{emb\,dim} 
\opn\codim{codim}

\opn\Tr{Tr} 
\opn\bigrank{big\,rank}
\opn\superheight{superheight}
\opn\lcm{lcm}
\opn\trdeg{tr\,deg}
\opn\reg{reg} 
\opn\lreg{lreg} 
\opn\ini{in} 
\opn\lpd{lpd}
\opn\size{size}
\opn\mult{mult}
\opn\dist{dist}
\opn\cone{cone}
\opn\lex{lex}
\opn\rev{rev}
\opn\div{div} \opn\Div{Div} \opn\cl{cl} \opn\Cl{Cl}
\opn\Spec{Spec} \opn\Supp{Supp} \opn\supp{supp} \opn\Sing{Sing}
\opn\Ass{Ass} \opn\Min{Min}
\opn\Ann{Ann} \opn\Rad{Rad} \opn\Soc{Soc}
\opn\Syz{Syz} \opn\Im{Im} \opn\Ker{Ker} \opn\Coker{Coker}
\opn\Am{Am} \opn\Hom{Hom} \opn\Tor{Tor} \opn\Ext{Ext}
\opn\End{End} \opn\Aut{Aut} \opn\id{id} \opn\ini{in}

\opn\nat{nat}
\opn\pff{pf}
\opn\Pf{Pf} \opn\GL{GL} \opn\SL{SL} \opn\mod{mod} \opn\ord{ord}
\opn\Gin{Gin}
\opn\Hilb{Hilb}\opn\adeg{adeg}\opn\std{std}\opn\ip{infpt}
\opn\Pol{Pol}
\opn\sat{sat}
\opn\Var{Var}
\opn\Gen{Gen}

\opn\aff{aff} \opn\con{conv} \opn\relint{relint} \opn\st{st}
\opn\lk{lk} \opn\cn{cn} \opn\core{core} \opn\vol{vol}
\opn\link{link} \opn\star{star}
\opn\gr{gr}

\def\pot#1#2{#1[\kern-0.28ex[#2]\kern-0.28ex]}

\opn\dirlim{\underrightarrow{\lim}}
\opn\inivlim{\underleftarrow{\lim}}
\let\to=\rightarrow

\def\Implies{\ifmmode\Longrightarrow \else
        \unskip${}\Longrightarrow{}$\ignorespaces\fi}
\def\implies{\ifmmode\Rightarrow \else
        \unskip${}\Rightarrow{}$\ignorespaces\fi}
\def\iff{\ifmmode\Longleftrightarrow \else
        \unskip${}\Longleftrightarrow{}$\ignorespaces\fi}

\let\:=\colon
\newtheorem{Theorem}{Theorem}[section]
\newtheorem{Lemma}[Theorem]{Lemma}

\newtheorem{Proposition}[Theorem]{Proposition}

\newtheorem{Example}[Theorem]{Example}

\newtheorem{Observation}[Theorem]{Observation}
\let\epsilon\varepsilon
\let\phi=\varphi
\let\kappa=\varkappa
\def\qed{\ifhmode\textqed\fi
      \ifmmode\ifinner\quad\qedsymbol\else\dispqed\fi\fi}
\def\textqed{\unskip\nobreak\penalty50
       \hskip2em\hbox{}\nobreak\hfil\qedsymbol
       \parfillskip=0pt \finalhyphendemerits=0}
\def\dispqed{\rlap{\qquad\qedsymbol}}

\opn\dis{dis}
\opn\height{height}
\opn\dist{dist}
\def\pnt{{\raise0.5mm\hbox{\large\bf.}}}

\opn\Lex{Lex}

\begin{document}
\title{Nonincreasing depth functions of monomial ideals}
\author{Kazunori Matsuda, Tao Suzuki and Akiyoshi Tsuchiya}
\address{Kazunori Matsuda,
Department of Pure and Applied Mathematics,
Graduate School of Information Science and Technology,
Osaka University, Suita, Osaka 565-0871, Japan}
\email{kaz-matsuda@ist.osaka-u.ac.jp}
\address{Tao Suzuki,
Department of Pure and Applied Mathematics,
Graduate School of Information Science and Technology,
Osaka University, Suita, Osaka 565-0871, Japan}
\email{t-suzuki@ist.osaka-u.ac.jp}
\address{Akiyoshi Tsuchiya,
Department of Pure and Applied Mathematics,
Graduate School of Information Science and Technology,
Osaka University, Suita, Osaka 565-0871, Japan}
\email{a-tsuchiya@cr.math.sci.osaka-u.ac.jp}
\begin{abstract}
Given a nonincreasing  
function $f : \ZZ_{\geq 0} \setminus \{ 0 \} \to \ZZ_{\geq 0}$ such that
(i) $f(k) - f(k+1) \leq 1$ for all $k \geq 1$ and
(ii) if $a = f(1)$ and $b = \lim_{k \to \infty} f(k)$, then
$|f^{-1}(a)| \leq |f^{-1}(a-1)| \leq \cdots \leq |f^{-1}(b+1)|$,
a system of generators of a monomial ideal $I \subset K[x_1, \ldots, x_n]$ 
for which $\depth S/I^k = f(k)$ for all $k \geq 1$ is explicitly described.  
Furthermore, we give a characterization of triplets of integers $(n,d,r)$ 
with $n > 0$, $d \geq 0$ and $r > 0$ with the properties that 
there exists a monomial ideal $I \subset S = K[x_1, \ldots, x_n]$
for which $\lim_{k \to \infty} \depth S/I^k = d$ and
${\rm dstab}(I) = r$, where ${\rm dstab}(I)$ is the smallest integer 
$k_0 \geq 1$ 
with $\depth S/I^{k_0} = \depth S/I^{k_0+1} = \depth S/I^{k_0+2} = \cdots$.
\end{abstract}
\maketitle
\section*{Introduction}
The study on depth of powers of ideals, 
which originated in \cite{HH},
has been achieved by many authors in the last decade.
Let $S = K[x_1, \ldots, x_n]$ denote the polynomial ring
in $n$ variables over a field $K$ and $I \subset S$ a homogeneous ideal.
The numerical function $f : \ZZ_{\geq 0} \setminus \{ 0 \} \to \ZZ_{\geq 0}$
defined by $f(k) = \depth S/I^k$ is called the {\em depth function} of $I$.
It is known \cite{B} that $f(k) = \depth S/I^k$ is constant for $k \gg 0$.
We call $\lim_{k \to \infty} f(k)$ the {\em limit depth} of $I$.
The smallest integer $k_0 \geq 1$ for which
$f(k_0) = f(k_0 + 1) = f(k_0 + 2) = \cdots$ 
is said to be the {\em depth stability number} of $I$ 
and is denoted by ${\rm dstab}(I)$.

An exciting conjecture (\cite[p.~549]{HH})  
is that {\em any} convergent function 
$f : \ZZ_{\geq 0} \setminus \{ 0 \} \to \ZZ_{\geq 0}$
can be the depth function of a homogeneous ideal.
In \cite[Theorem 4.1]{HH}, 
given a bounded nondecreasing function
$f : \ZZ_{\geq 0} \setminus \{ 0 \} \to \ZZ_{\geq 0}$,
a system of generators of a monomial ideal $I$ for which 
$\depth S/I^k = f(k)$ for all $k \geq 1$
is explicitly described.
In \cite[Theorem 4.9]{HTT}, it is shown that,
given a nonincreasing function
$f : \ZZ_{\geq 0} \setminus \{ 0 \} \to \ZZ_{\geq 0}$, 
there exists a monomial ideal $Q$ for which $\depth S/Q^k = f(k)$ 
for all $k \geq 1$.  Unlike the proof of \cite[Theorem 4.1]{HH},
since the proof of \cite[Theorem 4.9]{HTT} relies on induction on 
$\lim_{k \to \infty} f(k)$, no explicit description of a system of generators 
of a monomial ideal $Q$ is provided.

Our original motivation to organize this paper 
was to find an explicit description of 
a system of generators of a monomial ideal $Q$ of \cite[Theorem 4.9]{HTT}.
However, there seems to be a gap in the proof of \cite[Theorem 4.9]{HTT}
and it is unclear whether \cite[Theorem 4.9]{HTT} is true.
In fact, the inductive argument done in the proof of \cite[Theorem 4.9]{HTT}
cannot be valid for the nonincreasing function 
$f : \ZZ_{\geq 0} \setminus \{ 0 \} \to \ZZ_{\geq 0}$ 
with $f(1) = f(2) = 2$ and $f(3) = f(4) = \cdots = 0$. 
In the present paper, given a nonincreasing  
function $f : \ZZ_{\geq 0} \setminus \{ 0 \} \to \ZZ_{\geq 0}$ such that
\begin{itemize}
\item
$f(k) - f(k+1) \leq 1$ for all $k \geq 1$;
\item
if $a = f(1)$ and $b = \lim_{k \to \infty} f(k)$, then
\[
|f^{-1}(a)| \leq |f^{-1}(a-1)| \leq \cdots \leq |f^{-1}(b+1)|,
\]
\end{itemize}
a system of generators of a monomial ideal $I$ for which 
$\depth S/I^k = f(k)$ for all $k \geq 1$ is explicitly described
(Theorem \ref{main}).  
Furthermore, we give a characterization of triplets of integers $(n,d,r)$ 
with $n > 0$, $d \geq 0$ and $r > 0$ with the properties that 
there exists a monomial ideal $I \subset S = K[x_1, \ldots, x_n]$
for which $\lim_{k \to \infty} \depth S/I^k = d$ and
${\rm dstab}(I) = r$ (Theorem \ref{ndr}).

\section{Nonincreasing depth functions}
 Let $K$ be a field and $S = K[x_1,\ldots,x_n]$ the polynomial ring in $n$ variables over $K$ with
each $\deg x_i = 1$.

In this section, we show the following theorem.

\begin{Theorem}
	\label{main}
	Given a nonincreasing  
	function $f : \ZZ_{\geq 0} \setminus \{ 0 \} \to \ZZ_{\geq 0}$ such that
	\begin{itemize}
		\item
		$f(k) - f(k+1) \leq 1$ for all $k \geq 1$;
		\item
		if $a = f(1)$ and $b = \lim_{k \to \infty} f(k)$, then
		\[
		|f^{-1}(a)| \leq |f^{-1}(a-1)| \leq \cdots \leq |f^{-1}(b+1)|,
		\]
	\end{itemize}
	there is a monomial ideal $I$ for which $\depth S/I^k = f(k)$ for all $k \geq 1$.
\end{Theorem}


At first, we prepare some lemmas to prove Theorem \ref{main}. 
\begin{Lemma}$($\cite[Corollary 5.11]{Nguyen}$)$
        \label{betti}
Let $I$ be a monomial ideal in $S$. 
Then for any integer $k \geq 1$, we have
$$ \depth I^{k-1}/I^k =  \min \{\depth I^{k-1} , \depth I^k - 1\}.$$
\end{Lemma}

\begin{Lemma}
	Let $I$ be a monomial ideal in $S$.
Then the following arguments are equivalent:
\begin{itemize}
	\label{noninc}
	\item[(a)] $\depth S/I^k$ is nonincreasing. 
	\item[(b)] $\depth I^{k-1}/I^{k}$ is nonincreasing.
\end{itemize}
Moreover, when this is the case, $\depth S/I^k=\depth I^{k-1}/I^{k}$ for any $k \geq 1$. 
\end{Lemma}

\begin{proof}
	Set $f(k)=\depth S/I^k$ and $g(k)=\depth I^{k-1}/I^{k}$.
Since we obtain $\depth I^k = \depth S/I^k +1$ for any $k \geq 1$,
 by Lemma  \ref{betti}, it is obvious that
$$g(k) = \min \{f(k-1) +1, f(k)\}, k=1,2,\ldots.$$
Hence we know that if $f(k)$ is nonincreasing, then we have $g(k) = f(k)$.

On the other hand, we assume that $g(k)$ is nonincreasing.
If $f(t)=g(t)$ for an integer $t \geq 1$,
then we have $f(t+1)=g(t+1)$.
Since $f(1)=g(1)$, it follows that for any integer $k \geq 1$, $f(k)=g(k)$.
\end{proof}

\begin{Lemma}
	\label{ideal}
	Set $A=K[x_1,\ldots,x_{n'}]$ and $B=K[x_{n'+1},\ldots,x_n]$, and we let $I$ , $J$ are monomial ideals in $A$ and $B$.
Then for any integer $t \geq 1$, we have
$$ \depth (I + J)^{t-1}/(I + J)^t = \min_{\stackrel{i+j = t+1}{i,j \geq 1}} \{\depth I^{i-1}/I^i + \depth J^{j-1}/J^j\}.$$
\end{Lemma}

\begin{proof}
	It follows by combining \cite[Theorem3.3 (i)]{HTT} and \cite[Theorem1.1]{Nguyen}.
\end{proof}

The following proposition is important in this paper.
\begin{Proposition}
	\label{01}
Let $t \geq 2$ be an integer and we set a monomial ideal $I = (x^t,xy^{t-2}z,y^{t-1}z)$ in $B = K[x,y,z]$. Then
\begin{displaymath}
\depth B/I^n =\left\{
\begin{aligned}
&1,& \ \textnormal{if}& \ n \leq t-1,\\
&0,& \ \textnormal{if}& \ n \geq t. 
\end{aligned}
\right.
\end{displaymath}
\end{Proposition}

\begin{proof}
First of all, for each integer $n \geq t$, we show that $\depth B/I^n = 0$. For this purpose we find a monomial belonging to $(I^n \colon \mm) \setminus I^n$, where $\mm = (x,y,z)$. 
We claim that the monomial $u = x^{tn-t^2+t}y^{t^2-2t}z^{t-1}$ belongs to $(I^n \colon \mm) \setminus I^n$.
Indeed, each generator of $I^n$ forms
$$
w(a,b,c) := (x^t)^a (xy^{t-2}z)^b (y^{t-1}z)^c = x^{ta+b}y^{(t-2)b+(t-1)c}z^{b+c},
$$
where $a+b+c = n$ and $a,b,c \geq 0$.
Then we have 
$$w(n-t+1,1,t-2)|xu,$$
$$ w(n-t+1,0,t-1)|yu,$$
$$w(n-t,t,0)|zu.$$
Thus $u \in  (I^n \colon \mm)$. 
While the degree of $u$ is less than that of generators in $I^n$.
Hence we obtain $u \notin I^n$.

Next, we show that $\pd I^n = 1$ for all $1 \leq n \leq t-1$. 
In order to prove this, we use the theory of \textit{Buchberger graphs}.
Let $m_1, \ldots, m_s$ be the generators of $I^n$.
The Buchberger graph $\textnormal{Buch}(I^n)$ has vertices $1,\ldots,s$ and an edge $(i,j)$ whenever there is no monomial $m_k$ such that $m_k$ divides $\textnormal{lcm}(m_i,m_j)$ and the degree of $m_k$ is different from  $\textnormal{lcm}(m_i,m_j)$ in every variable that occurs in $\textnormal{lcm}(m_i,m_j)$.
Then it is known that the syzygy module $\textnormal{syz}(I^n)$ is generated by syzygies 
$$\sigma_{ij}=\cfrac{\textnormal{lcm}(m_i,m_j)}{m_i}\ \eb_i -\cfrac{\textnormal{lcm}(m_i,m_j)}{m_j}\ \eb_j$$
 corresponding to edges $(i,j)$ in  $\textnormal{Buch}(I^n)$ (\cite[Proposition 3.5]{CCA}).

Let $G(I^n) := \{ w(a, b, c) = x^{ta+b}y^{(t-2)b+(t-1)c}z^{b+c} \mid a,b,c \ge 0, a + b + c = n \}$ be the set of generators of $I^n$. 
We introduce the following lexicographic order $<$ on $G(I^n)$.  
Let $w(a, b, c), w(a', b', c') \in G(I^n)$. 
Then we define
\begin{itemize}
	\item $w(a', b', c') < w(a, b, c)$ if $a' < a$; 
	\item $w(a', b', c') < w(a, b, c)$ if $a' = a$ and $b' < b$.
\end{itemize}

\begin{Observation}
\label{obs}
For $w = x^ay^bz^c$, we denote $\deg_{x}w = a$, $\deg_{y}w = b$ and 
$\deg_{z}w = c$.  
It is easy to see that 
\begin{itemize}
      \item $\deg_{x}w(a', b', c') < \deg_{x}w(a, b, c)$ if and only if $w(a', b', c') < w(a, b, c);$ 
      \item $\deg_{y}w(a', b', c') \ge \deg_{y}w(a, b, c)$ if $w(a', b', c') < w(a, b, c);$ 
      \item $\deg_{z}w(a', b', c') \ge \deg_{z}w(a, b, c)$ if $w(a', b', c') < w(a, b, c)$ 
\end{itemize}
if $1 \le n \le t - 1$. 
\end{Observation}

To construct the minimal free resolution of $I^n$, we compute generators of ${\rm syz} (I^n)$.  
For $w(a, b, c), w(a', b', c') \in G(I^n)$, we define $w(a', b', c') \lessdot w(a, b, c)$ if $w(a', b', c') < w(a, b, c)$ 
and there is no monomial $w \in G(I^n)$ such that $w(a', b', c') < w < w(a, b, c)$. 
Moreover, we put
\begin{eqnarray*}
& & \sigma ((a, b, c), (a', b', c')) \\
&:=& \cfrac{\textnormal{lcm}(w(a, b, c), w(a', b', c'))}{w(a, b, c)}\ \eb_{(a, b, c)} -\cfrac{\textnormal{lcm}(w(a, b, c), w(a', b', c'))}{w(a', b', c')}\ \eb_{(a', b', c')}. 
\end{eqnarray*}

We show that \\

{\bf Claim 1. } If $w(a', b', c') \lessdot w(a, b, c)$, then $\{w(a', b', c'), w(a, b, c) \}$ is an edge of ${\rm Buch} (I^n)$. \\

\noindent {\em Proof of Claim 1.} \ 
Note that $w(a', b', c') \lessdot w(a, b, c)$ if and only if either $a' = a, b' = b - 1$ and $c' = c + 1$ or $(a, b, c) = (a, 0, n - a)$ and $(a', b', c') = (a - 1, n - a + 1, 0)$. 
In the former case, we have
${\rm lcm}(w(a, b, c), w(a, b - 1, c + 1)) = x^{ta + b} y^{(t - 2)(b - 1) + (t - 1)(c + 1)} z^{n - a}$  from Observation \ref{obs}. 
It is enough to show that there is no monomial $w \in G(I^n)$ such that 
$w \mid {\rm lcm}(w(a, b, c), w(a, b - 1, c + 1)) / xyz = x^{ta + b - 1} y^{(t - 2)(b - 1) + (t - 1)(c + 1) - 1} z^{n - a - 1}$. 

Assume that there exists such a monomial $w \in G(I^n)$. 
Then $\deg_x w \le ta + b - 1$.  
Hence $w \le w(a, b - 1, c + 1)$ from Observation \ref{obs}. 
However, $\deg_z w \ge b + c = n - a$ from Observation \ref{obs} again, 
this is a contradiction. 

Next, we consider the latter case, that is, $(a, b, c) = (a, 0, n - a)$ and $(a', b', c') = (a - 1, n - a + 1, 0)$. 
As in the former case, it is enough to show that 
there is no monomial $w \in G(I^n)$ such that
$ w \mid {\rm lcm} (w(a, 0, n - a), w(a - 1, n - a + 1, 0)) / xyz = x^{ta - 1} y^{(t - 2)(n - a + 1) - 1} z^{n - a}$. 
Assume that there exists such a monomial $w \in G(I^n)$. 
Then $\deg_x w \le ta - 1$ and $w \le w(a - 1, n - a + 1, 0)$ from Observation \ref{obs}.    
But we have $\deg_z w \ge n - a + 1$ from Observation \ref{obs} again, 
this is a contradiction. 

Therefore, we have the desired conclusion. \qedhere \\

Here, we put $\Sigma := \{ \sigma ((a, b, c), (a', b', c')) \mid w(a', b', c') \lessdot w(a, b, c) \}$. 
Next, we will show the following: \\

{\bf Claim 2. } Assume that  $w(a', b', c') < w(a, b, c)$ and $w(a', b', c')\  / \!\!\!\!\!\lessdot w(a, b, c)$. 
Then $\sigma ((a, b, c), (a', b', c'))$ can be expressed as an $S$-linear combination of the elements of $\Sigma$. \\

\noindent {\em Proof of Claim 2.} \ 
Let $s \ge 3$ and assume that 
\[
w(a', b', c') = w(a_s, b_s, c_s) \lessdot w(a_{s-1}, b_{s-1}, c_{s-1}) \lessdot \cdots  \lessdot w(a_{1}, b_{1}, c_{1}) = w(a, b, c). 
\]
From Observation \ref{obs}, we can see that 
\[
\frac{{\rm lcm} (w(a_1, b_1, c_1), w(a_s, b_s, c_s))}{ {\rm lcm} (w(a_i, b_i, c_i), w(a_{i + 1}, b_{i + 1}, c_{i + 1})) }
\]
is a monomial in $S$ for all $1 \le i \le s - 1$. 
Hence we have
\begin{eqnarray*}
& & \sigma ((a, b, c), (a', b', c')) = \sigma ((a_1, b_1, c_1), (a_s, b_s, c_s)) \\
&=& \sum_{i = 1}^{s - 1} \frac{{\rm lcm} (w(a_1, b_1, c_1), w(a_s, b_s, c_s))}{ {\rm lcm} (w(a_i, b_i, c_i), w(a_{i + 1}, b_{i + 1}, c_{i + 1})) } \sigma ((a_i, b_i, c_i), (a_{i + 1}, b_{i + 1}, c_{i + 1})). 
\end{eqnarray*} 
Thus we have the desired conclusion. \qed \\

By Claim 1, 2 and \cite[Proposition 3.5]{CCA}, $\Sigma$ is the set of generators 
of ${\rm syz}(I^n)$. 
Moreover, it is clear that the elements of $\Sigma$ are linearly independent on $S$. 
Hence 
\[
0 \to \bigoplus_{j} S(-j)^{\beta_{1, j}} \to S(-nt)^{\beta_{0, nt}} \to I^n \to 0
\]
is the minimal free resolution of $I^n$. 
Therefore we have $\pd I^n = 1$. \qed \\
\end{proof}

Now, we can prove Theorem \ref{main}. \\
\begin{proof}[Proof of Theorem \ref{main}]
First, for any integers $i,k \geq 1$, we define the monomial ideal
$I_{k, i} := (x_i^{k+1},x_iy_i^{k-1}z_i,y_i^kz_i)$
in $B_i=K[x_i,y_i,z_i]$.
Then by Proposition \ref{01}, we obtain
\begin{displaymath}
\depth B_i/I_{k,i}^t =\left\{
\begin{aligned}
&1,& \ \textnormal{if}& \ t \leq k,\\
&0& \ \textnormal{if}& \ t > k.
\end{aligned}
\right.
\end{displaymath}
Set $n = a-b$ and $s_i := |f^{-1}(a-i+1)|$ for each $1 \leq i \leq n$. 
We show that 
$I = \sum_{i =1}^{n} I_{s_i,i}$
in $S = K[x_{1},y_{1},z_{1},\dots ,x_{n},y_{n},z_{n},w_1, \dots ,w_b]$ 
is the required monomial ideal. 
By Lemma \ref{noninc} and \ref{ideal}, we immediately show the assertion follows. 
\end{proof} 

\begin{Example}
Nonincreasing functions $f : \ZZ_{\geq 0} \setminus \{ 0 \} \to \ZZ_{\geq 0}$ with $f(1) = f(2) = 2$ 
and $f(3) = f(4) = \cdots = 0$ and $g : \ZZ_{\geq 0} \setminus \{ 0 \} \to \ZZ_{\geq 0}$ with 
$g(1) = g(2) = 2$, $g(3) = 1$ and $g(4) = g(5) = \cdots = 0$  do not satisfy the assumption of 
Theorem \ref{main}. 
However there exist monomial ideals $I$, $J$ of $S = K[x_1, \ldots, x_6]$ such that $\depth S/I^k = f(k)$ and $\depth S/J^k = g(k)$ for $k \ge 1$. 

Indeed, $I = (x_1^3, x_1x_2x_3, x_2^2x_3)(x_4^3, x_4x_5x_6, x_5^2x_6) + (x_1^4, x_1^3x_2, x_1x_2^3, x_2^4, x_1^2x_2^2x_3)$ and $J = (x_1^4, x_1x_2^2x_3, x_2^3x_3)(x_4^4, x_4x_5^2x_6, x_5^3x_6) + (x_1^5, x_1^4x_2, x_1x_2^4, x_2^5, x_1^3x_2^2x_3)$ 
are the desired monomial ideals. 
\end{Example}

\section{the number of variables and depth stability number}
 	Let $I \neq (0)$ be a monomial ideal in $S = K[x_1, \ldots, x_n]$ and $f(k)$ the depth function of $I$.
 	We set $\lim_{k\rightarrow \infty}f(k)=d$ and $r=\textnormal{dstab}(I)$.
 	When $n=1$, we know that $d=0$ and $r=1$.
 	Moreover, when $n=2$, we have $0 \leq d \leq 1$ and $r=1$.
 	
 	In this section, for $n \geq 3$, we discuss bounds of the limit depth and depth stability number of a monomial ideal.
 	In fact, we show the following theorem.
\begin{Theorem}\label{ndr}
	Assume $n \geq 3$.
	Let $I \neq (0)$ be a monomial ideal in $S = K[x_1, \ldots, x_n]$ and $f(k)$ the depth function of $I$.
	We set $\lim_{k\rightarrow \infty}f(k)=d$ and $r=\textnormal{dstab}(I)$.
	Then one of the followings is satisfied:
\begin{itemize}
	\item $0 \leq d \leq n-2$ and $r \geq 1$.
	\item $d=n-1$ and $r=1$.
\end{itemize}	
Conversely, for any $d$ and $r$ satisfied one of the above, there exists a monomial ideal $J$ in $S$ such that  $\lim_{k\rightarrow \infty}g(k)=d$ and $r=\textnormal{dstab}(J)$, where $g(k)$ is the depth function of $J$.
\end{Theorem}
\begin{proof}
	In general, for any monomial ideal $I \neq (0)$ in $S$, we have
	$0 \leq \depth S/I \leq n-1$.
	We assume that $d=n-1$.
	Since $\dim S/I^r \leq n-1$, $S/I^r$ is Cohen-Macaulay.
	Hence for any minimal prime ideal $P$ of $I^r$, we have $\height P=1$.
	In particular, $P$ is a principle ideal since $S$ is UFD.
	Hence $I^r$ is a principle ideal.
	This says that $I$ is also a principle ideal.
	Thus, for any $k \geq 1$, $S/I^k$ is a hypersurface.
	Therefore, we have $r=1$.
	
	Next, we show the latter part.
	Assume that $0 \leq d \leq n-3$ and $r \geq 2$.
	Let $J_1 = (x_1^r, x_1x_2^{r-2}x_3, x_2^{r-1}x_3) \subset A := K[x_1, x_2, x_3]$. 
	By Proposition \ref{01}, we have
	\begin{displaymath}
	\depth A/J_1^k =\left\{
	\begin{aligned}
	&0,& \ \textnormal{if}& \ k \geq r,\\
	&1,& \ \textnormal{if}& \ k \leq r-1.
	\end{aligned}
	\right.
	\end{displaymath}
Let $J=J_1 + (x_4, \ldots, x_{n - d}) = (x_1^r, x_1x_2^{r-2}x_3, x_2^{r-1}x_3, x_4, \ldots, x_{n-d})$ be a monomial ideal in $S$
and $g_1(k)$ the depth function of $J$.
Then we have $\lim_{k\rightarrow \infty}g_1(k)=d$ and $\textnormal{dstab}(J)=r$.
Moreover, an ideal $J_2 = (x_1, \ldots x_{n - d}) \subset S$ satisfies that 
$\depth (S/J_2^k) = d$ for all $k \ge 1$, that is, $\lim_{k\rightarrow \infty} \depth (S/J_2^k) =d$ and $ \textnormal{dstab}(J_2)=1$. 

Next, we assume that $d=n-2$ and $r \geq 1$.
By \cite[Proof of Theorem 4.1]{HH}, we can see that 
a monomial ideal $J_3 = (x_1^{r + 2}, x_1^{r + 1}x_2, x_1x_2^{r +1}, x_2^{r + 2}, x_1^rx_2^2x_3) \subset A$ satisfies that $\textnormal{dstab}(J_3)=r$ and
\begin{displaymath}
\depth A/J_3^k =\left\{
\begin{aligned}
&1,& \ \textnormal{if}& \ k \geq r,\\
&0,& \ \textnormal{if}& \ k \leq r-1.
\end{aligned}
\right.
\end{displaymath}
Let $J'=J_3$ be the monomial ideal in $S$ and $g_2(k)$ the depth function of $J'$. 
Then we have $\lim_{k\rightarrow \infty}g_2(k)=d$ and $\textnormal{dstab}(J')=r$.

When $d=n-1$ and $r=1$,
we immedietly obtain a monomial ideal satisfied the condition by the former part of this proof, as desired. 
\end{proof}

\end{document}